\documentclass[reqno, 11pt]{amsart}

\usepackage{amsmath,amssymb,amsthm}
\usepackage{mathrsfs}
\usepackage{dsfont}
\usepackage{mathabx}\changenotsign

\usepackage{fmtcount}
\usepackage{microtype}

\usepackage{xcolor}
\usepackage[backref=page]{hyperref}
\hypersetup{%
    colorlinks,
    linkcolor={red!50!black},
    citecolor={green!50!black},
    urlcolor={blue!50!black}
}

\usepackage{varioref}

\usepackage[english,algoruled,vlined,resetcount,linesnumbered]{
algorithm2e}%,linesnumbered

\usepackage{bookmark}
\usepackage{mathtools}
\usepackage[abbrev,msc-links,backrefs]{amsrefs}
\usepackage{doi}

\usepackage[utf8]{inputenc}

\renewcommand{\PrintDOI}[1]{\doi{#1}}

\usepackage[T1]{fontenc}
\usepackage{lmodern}

\usepackage[english]{babel}

\usepackage{tikz}

\usepackage{fullpage}
\usepackage{setspace}
\makeatletter
\def\@setdate{\datename\ \@date}
\def\@setaddresses{\par
  \nobreak \begingroup
\footnotesize
  \def\author##1{\nobreak\addvspace\bigskipamount}%
  \def\\{\unskip, \ignorespaces}%
  \interlinepenalty\@M
  \def\address##1##2{\begingroup
    \par\addvspace\bigskipamount\indent
    \@ifnotempty{##1}{(\ignorespaces##1\unskip) }%
    {\scshape\ignorespaces##2}\par\endgroup}%
  \def\curraddr##1##2{\begingroup
    \@ifnotempty{##2}{\nobreak\indent{\itshape Current address}%
      \@ifnotempty{##1}{, \ignorespaces##1\unskip}\/:\space
      ##2\par}\endgroup}%
  \def\email##1##2{\begingroup
    \@ifnotempty{##2}{\nobreak\indent{\itshape Email address}%
      \@ifnotempty{##1}{, \ignorespaces##1\unskip}\/:\space
      \ttfamily##2\par}\endgroup}%
  \def\urladdr##1##2{\begingroup
    \@ifnotempty{##2}{\nobreak\indent{\itshape URL}%
      \@ifnotempty{##1}{, \ignorespaces##1\unskip}\/:\space
      \ttfamily##2\par}\endgroup}%
  \addresses
  \endgroup
}
\makeatother
\usepackage{enumitem}

\let\polishlcross=\l
\def\l{\ifmmode\ell\else\polishlcross\fi}

\renewcommand{\emptyset}{\varnothing}
\renewcommand{\setminus}{\smallsetminus}

\makeatletter
\def\moverlay{\mathpalette\mov@rlay}
\def\mov@rlay#1#2{\leavevmode\vtop{%
   \baselineskip\z@skip \lineskiplimit-\maxdimen
   \ialign{\hfil$\m@th#1##$\hfil\cr#2\crcr}}}
\newcommand{\charfusion}[3][\mathord]{
    #1{\ifx#1\mathop\vphantom{#2}\fi
        \mathpalette\mov@rlay{#2\cr#3}
      }
    \ifx#1\mathop\expandafter\displaylimits\fi}
\makeatother

\newcommand{\eps}{\varepsilon}
\let\epsilon\varepsilon
\let\tilde\widetilde
\let\ldots\dots

\newtheoremstyle{case}{}{}{\normalfont}{}{\itshape}{:}{ }{}

\newtheorem{thm}{Theorem}
\newtheorem{lem}[thm]{Lemma}
\newtheorem{prop}[thm]{Proposition}
\newtheorem{conj}[thm]{Conjecture}

\newtheorem*{claim*}{Claim}

\theoremstyle{definition}
\newtheorem{defn}[thm]{Definition}

\newtheorem{ex}[thm]{Example}

\newtheoremstyle{case}{}{}{\normalfont}{}{\itshape}{\normalfont:}{ }{}

\theoremstyle{case}

\newcommand{\oldqed}{}
\def\endofClaim{\hfill\scalebox{.6}{$\Box$}}

\newenvironment{claimproof}[1][Proof]{
  \renewcommand{\oldqed}{\qedsymbol}
  \renewcommand{\qedsymbol}{\endofClaim}
  \begin{proof}[#1]
}{
  \end{proof}
  \renewcommand{\qedsymbol}{\oldqed}
}

\let\:\colon
\let\subset\subseteq

\def\({\left(}
\def\){\right)}
\let\<\langle
\let\>\rangle

\newcommand{\PP}{\mathbb{P}}

\def\cF{\mathcal{F}}

\def\cG{\mathcal{G}}
\def\cT{\mathcal{T}}

\usepackage{datetime}
\usepackage{lineno}
\newcommand*\patchAmsMathEnvironmentForLineno[1]{%
\expandafter\let\csname old#1\expandafter\endcsname\csname #1\endcsname
\expandafter\let\csname oldend#1\expandafter\endcsname\csname end#1\endcsname
\renewenvironment{#1}%
{\linenomath\csname old#1\endcsname}%
{\csname oldend#1\endcsname\endlinenomath}}%
\newcommand*\patchBothAmsMathEnvironmentsForLineno[1]{%
\patchAmsMathEnvironmentForLineno{#1}%
\patchAmsMathEnvironmentForLineno{#1*}}%
\AtBeginDocument{%
\patchBothAmsMathEnvironmentsForLineno{equation}%
\patchBothAmsMathEnvironmentsForLineno{align}%
\patchBothAmsMathEnvironmentsForLineno{flalign}%
\patchBothAmsMathEnvironmentsForLineno{alignat}%
\patchBothAmsMathEnvironmentsForLineno{gather}%
\patchBothAmsMathEnvironmentsForLineno{multline}%
}
\begin{document}
%\linenumbers
\onehalfspace
%\doublespace

%TITLE, ETC
\title{2-universality in randomly perturbed graphs}
\author[O. Parczyk]{Olaf Parczyk}

\thanks{The author was partially supported by the Carl Zeiss Foundation.
}

%\shortdate
%\yyyymmdddate
%\settimeformat{ampmtime}
%\date{\today, \currenttime}
%\footskip=28pt

\address{Institut f\"ur Mathematik, TU Ilmenau,
  Weimarer Str.~25, 98684 Ilmenau, Germany}
\email{olaf.parczyk@tu-ilmenau.de}

\begin{abstract}
 
 A graph $G$ is called universal for a family of graphs $\cF$ if it contains every element $F \in \cF$ as a subgraph.
 Let $\cF(n,2)$ be the family of all graphs with maximum degree $2$.
 Ferber, Kronenberg, and Luh~[\emph{Optimal Threshold for a Random Graph to be 2-Universal}, to appear in Transactions of the American Mathematical Society] proved that there exists a $C$ such that for $p \ge C (n^{-2/3} \log^{1/3} n )$ the random graph $G(n,p)$ a.a.s~is $\cF(n,2)$-universal, which is asymptotically optimal.
 For any $n$-vertex graph $G_\alpha$ with minimum degree $\delta(G_\alpha) \ge \alpha n$ Aigner and Brandt~[\emph{Embedding arbitrary graphs of maximum degree two}, Journal of the London Mathematical Society \textbf{48} (1993), 39--51] proved that $G_\alpha$ is $\cF(n,2)$-universal for an optimal $\alpha \ge 2/3$.
 
 In this note, we consider the model of randomly perturbed graphs, which is the union $G_\alpha \cup G(n,p)$.
 We prove that a.a.s.~$G_\alpha \cup G(n,p)$ is $\cF(n,2)$-universal provided that $\alpha>0$ and $p=\omega(n^{-2/3})$.
 This is asymptotically optimal and improves on both results from above in the respective parameter.
 Furthermore, this extends a result of Böttcher, Montgomery, Parczyk, and Person~[\emph{Embedding spanning bounded degree subgraphs in randomly perturbed graphs}, arXiv:1802.04603 (2018)], who embed a given $F \in \cF(n,2)$ at these values.
 We also prove variants with universality for the family $\cF^\ell(n,2)$, all graphs from $\cF(n,2)$ with girth at least $\ell$.
 For example, there exists an $\ell_0$ depending only on $\alpha$ such that for all $\ell \ge \ell_0$ already $p=\omega(1/n)$ is sufficient for $\cF^\ell(n,2)$-universality.
\end{abstract}

\maketitle

\section{Introduction}

Finding cyclic structures is a natural problem in graph theory.
A famous result of Dirac~\cite{dirac1952some} states that any $n$-vertex graph $G$ with minimum degree $\delta(G) \ge n/2$ contains a Hamilton cycle if $n \ge 3$.
More generally, in extremal graph theory many results are of the form that for a sequence of $n$-vertex graphs $(F_n)_{n \ge 1}$ there exists some $\alpha>0$ such that any $n$-vertex graph $G_\alpha$ with $\delta(G_\alpha) \ge \alpha n$ contains $F_n$.

Investigating this type of containment questions for a typical graph gives us a different perspective.
In random graph theory properties of the model $G(n,p)$ are studied, which is the $n$-vertex \emph{binomial random graph} with each edge present independently at random with probability $p$.
The counterpart to the condition on the minium degree in this setup is a function $\hat{p} = \hat{p}(n) \colon \mathbb{N} \rightarrow [0,1]$ such that for $p=\omega(\hat{p})$ the probability that the $n$-vertex graph $F_n$ is contained in $G(n,p)$ tends to $1$ as $n$ tends to infinity, i.e.~$ \lim_{n \to \infty} \PP [F_n \subseteq G(n,p)]=1$.
In this case we say that $G(n,p)$ contains $F_n$ asymptotically almost surely (a.a.s.).
If $\hat{p}$ is optimal in the sense that $\lim_{n \to \infty} \PP [F_n \subseteq G(n,p)]=0$ for $p=o(\hat{p})$, we call $\hat{p}$ the \emph{threshold} for the property of containing $F_n$.
Bollobás and Thomason~\cite{BolTho87} proved that monotone properties, as subgraph containment, always admit a threshold.
For containing a Hamilton cycle Kor{\v{s}}unov~\cite{Kor76} and Pósa~\cite{Pos76} independently showed that the threshold is $\log n/n$.
Note that the expected number of Hamilton cycles already tends to infinity for $p = \omega(1/n)$, but the extra $\log n$-term ist needed to guarantee that the graph has a.a.s.~minimum degree $2$.
In the following discussion we will work with $p=\omega(\hat{p})$ even though for many results a stronger variant is proved, where $p \ge C \hat{p}$ for some absolute constant $C$ depending only on $(F_n)_{n \ge 1}$.

\subsection{Randomly perturbed graphs}

Combining the two models from random and extremal graph theory, Bohman, Frieze, and Martin~\cite{bohman2003many} introduced the model of \emph{randomly perturbed graphs} $G_\alpha \cup G(n,p)$ for any $\alpha >0$, where, as above, $G_\alpha$ is any $n$-vertex graph with minimum degree $\delta(G_\alpha) \ge \alpha n$.
They show that $p = \omega(1/n)$ is sufficient to a.a.s.~guarantee a Hamilton cycle in $G_\alpha \cup G(n,p)$ for any $G_\alpha$.
When $G_\alpha$ is the complete unbalanced bipartite graph $K_{\alpha n,(1-\alpha)n}$ then at least a linear number of egdes is needed.
Using both graphs $G_\alpha$ and $G(n,p)$ this result dramatically improves on the $\alpha\ge1/2$ needed in $G_\alpha$ alone, even though adding $G(n,p)$ is a relatively small random perturbation.
On the other hand it is also a $\log n$-term better than the threshold in $G(n,p)$ alone, which is plausible as $G_\alpha$ guarantees a large minimum degree.

In recent years this model attracted a lot of attention.
For a bounded degree spanning tree Krivelevich, Kwan, and Sudakov~\cite{krivelevich2015bounded} proved that $p=\omega(1/n)$ also is sufficient in $G_\alpha \cup G(n,p)$.
In $G_\alpha$ alone $\alpha > 1/2$ is needed~\cite{KSS_AlonYuster} and in $G(n,p)$ only recently Montgomery~\cite{M14a} was able to show that again $\log n /n$ gives the threshold.
Similar results were proved for factors in~\cite{BTW_tilings} and for powers of Hamilton cycles and general bounded degree graphs in~\cite{BMPPMinDegree}.
Together with Böttcher, Montgomery, and Person in~\cite{BMPPMinDegree} we developed a general method for embeddings in randomly peturbed graphs that also implies the previous results for the Hamilton cycle, factors, and bounded degree trees.
For all these graphs $G_\alpha = K_{\alpha n,(1-\alpha)n}$ shows that the results are optimal.
Beyond this there are many interesting results for other properties~\cites{bohman2004adding,DT_ramsey}, trees with large degrees~\cite{JK_Trees}, hypergraphs~\cites{krivelevich2015cycles,MM_HyperPeturbed,HZ_perturbedcycles,BHKM_powers}, and with larger bounds on $\alpha$~\cites{RSR08_Dirac,NT_sprinkling}.

\subsection{Universality}

We call a graph $G$ \emph{universal} for a family of graphs $\cF$ (short $\cF$-universal) if $G$ contains every $F \in \cF$ as a subgraph.
In the random graph model when $\cF$ is large it requires substantial more work to prove that $G(n,p)$ is a.a.s.~$\cF$-universal, then showing that a given $F \in \cF$ is a.a.s.~contained in $G(n,p)$.
For the family $\cT(n,\Delta)$ of all $n$-vertex graphs with maximum degree bounded by $\Delta$, Montgomery~\cite{M14a} managed to prove that a.a.s.~$G(n,p)$ is $\cT(n,\Delta)$-universal if $p=\omega(\log n/n)$.
Krivelevich, Kwan, and Sudakov~\cite{krivelevich2015bounded} asked if extending on their result also $\cT(n,\Delta)$-universality holds in $G_\alpha \cup G(n,p)$ for $p=\omega(1/n)$ and $\alpha>0$.
In~\cite{BHKMPPtrees} we proved this together with Böttcher, Han, Kohayakawa, Montgomery, and Person building on the method from~\cite{BMPPMinDegree}.

In this note we want to investigate universality with respect to the family $\cF(n,2)$, which contains all $n$-vertex graphs with maximum degree at most $2$.
Graphs in this family are the disjoint union of paths and cycles.
In $G_\alpha$ alone Aigner and Brandt~\cite{AB_maxdegree2} showed that $\alpha\ge  2/3$ ($\delta(G_\alpha) \ge (2n-1)/3$ to be precise) is sufficient to find any graph from $\cF(n,2)$.
On the other side it was proved by Ferner, Kronenberg, and Luh~\cite{ferber2016optimal} that with $p = \omega(n^{-2/3} \log^{1/3}n)$ a.a.s.~$G(n,p)$ is $\cF(n,2)$-universal.
The disjoint union of $n/3$ triangles ($K_3$-factor) is seemingly the hardest graph to embed and the threshold, which follows from a famous result of Johannson, Kahn, and Vu~\cite{JohanssonKahnVu_FactorsInRandomGraphs}, shows that this is optimal.
In~\cite{BMPPMinDegree} we already proved that for a given $F \in \cF(n,2)$ it is a.a.s.~contained in $G_\alpha \cup G(n,p)$ for $p=\omega(n^{-2/3})$ and $\alpha>0$, and in this paper we extend this to $\cF(n,2)$-universality.
The following is our main result, which is asymptotically optimal when $\alpha < 1/3$.

\begin{thm}\label{thm:2-universal}
  Let $\alpha>0$, $p = \omega(n^{-2/3})$, and $G_\alpha$ an $n$-vertex graph with minimum degree $\delta(G_\alpha) \ge \alpha n$.
  Then a.a.s.~$G_\alpha \cup G(n,p)$ is $\cF(n,2)$-universal.
\end{thm}

In fact, as in~\cite{ferber2016optimal} we prove a stronger statement when there are no short cycles.
Let $\cF^\ell(n,2)$ be the family of all graphs with maximum degree at most $2$ and girth at least $\ell$, which implies that there are no cycles of length less than $\ell$.

\begin{thm}\label{thm:girth_2-universal}
	Let $\alpha>0$, $\ell \ge 3$ an integer, $p = \omega(n^{-(\ell-1)/\ell})$, and $G_\alpha$ an $n$-vertex graph with minimum degree $\delta(G_\alpha) \ge \alpha n$.
	Then a.a.s.~$G_\alpha \cup G(n,p)$ is $\cF^{\ell}(n,2)$-universal.
\end{thm}

Here the bound on $p$ is determined by the threshold of an almost spanning $C_\ell$-factor in $G(n,p)$ (c.f.~Lemma~\ref{lem:cycle}).
Theorem~\ref{thm:2-universal} follows from Theorem~\ref{thm:girth_2-universal} with $\ell=3$.
When the bound on the girth gets large enough in terms of $\alpha$ we can further improve this and show that $p = \omega (1/n)$ is enough.

\begin{thm}\label{thm:large_2-universal}
	For every $\alpha > 0$ there exists an $\ell_0>0$ such that for any integer $\ell \ge \ell_0$ and $G_\alpha$ an $n$-vertex graph with minimum degree $\delta(G_\alpha) \ge \alpha n$ the following holds.
	For $p = \omega(1/n)$ a.a.s.~$G_\alpha \cup G(n,p)$ is $\cF^\ell(n,2)$-universal.
\end{thm}

Our proof roughly gives $\ell_0 = 10^6 \alpha^{-3}$, where me make no effort in optimizing the constant in front of $\alpha^{-3}$.
The optimal dependence between $\ell_0$ and $\alpha$ remains open, where $\ell_0 \ge \alpha^{-1}$ follows from $G_\alpha = K_{\alpha n,(1-\alpha)n}$.
From the proof of Ferber, Kronenberg, and Luh~\cite{ferber2016optimal} together with a connecting result by Montgomery~\cite[Theorem 4.3]{M14b} one can deduce the following in $G(n,p)$ alone.
For $\ell \ge 10^4 \log^2 n$ and $p \ge \log^5 n /n$ a.a.s.~$G(n,p)$ is $\cF^\ell(n,2)$-universal.

For larger $\Delta$ the threshold for the $K_{\Delta+1}$-factor in $G(n,p)$ is $(\log^{1/(\Delta+1)}n /n)^{2/(\Delta+1)}$~\cite{JohanssonKahnVu_FactorsInRandomGraphs} and it is commonly believed that this also gives the threshold for $\cF(n,\Delta)$-universality, where here $\cF(n,\Delta)$ is the family of all $n$-vertex graphs with maximum degree $\Delta$.
But until now not even the single containment case is solved and the best known result is an almost spanning version by Ferber, Luh, and Nguyen~\cite{ferber2016embedding} showing that for $F \in \cF((1-\varepsilon)n,\Delta)$ with $\Delta \ge 5$ and $p=\omega(\log^{1/(\Delta+1)}n /n)^{2/(\Delta+1)}$ a.a.s.~$G(n,p)$ contains a copy of $F$. 
For spanning universality Ferber and Nenadov~\cite{FerbeNenadovSpanning} proved that $p=\omega(\log^3 n / n)^{1/(\Delta-1/2)}$ is sufficient, which is just below a natural barrier of $(\log n/n)^{1/\Delta}$ that was known before~\cite{DKRR15} and still a polynomial away from the lower bound.
The fact that an $F \in \cF(n,\Delta)$ with $\Delta \ge 5$ and $p=\omega(n^{-2/(\Delta+1)})$ is a.a.s.~contained in $G_\alpha \cup G(n,p)$~\cite{BMPPMinDegree}, together with Theorem~\ref{thm:2-universal}, support the following conjecture.
\begin{conj}
  Let $\alpha>0$, $\Delta \ge 3$ an integer, $p = \omega(n^{-2/(\Delta+1)})$, and $G_\alpha$ an $n$-vertex graph with minimum degree $\delta(G_\alpha) \ge \alpha n$.
  Then a.a.s.~$G_\alpha \cup G(n,p)$ is $\cF(n,\Delta)$-universal.
\end{conj}
In a subsequent paper we will prove this conjecture for $\Delta=3$.

\subsection{Proofsketch}

The proof of our results essentially follows the approach in~\cite{BHKMPPtrees} but instead of trees we now have to work with a union of cycles.
To obtain a universality result we show that a.a.s.~$G(n,p)$ satisfies certain pseudorandom properties (c.f.~Definition~\ref{def:graphs} and Proposition~\ref{prop:expander}) and, therefore, we can work in a deterministic graph $H=G_\alpha \cup G$.
From $F \in \cF^\ell(n,2)$ in Step~\ref{step1} we embed a small, but linear sized, subgraph $F'$ of $F$ into $G$ (c.f.~Lemma~\ref{lem:2-partition} and~\ref{lem:embeddingF_1}) with the additional property that for any vertex $v \in V(H)$ there is a linear set $B(v)$ of vertices which can be replaced by $v$ in the embedding of $F'$.
Furthermore any other vertex $u \in V(G)$ has linear degree into all the sets $B(v)$.
This reservoir property is essential for the final step of our embedding process.
In Step~\ref{step2} we extend the embedding of $F'$ to an embedding containing all but $\varepsilon n$ vertices of $F$.
This is easy, because we only have to embed small cycles and some longer paths (c.f.~Lemma~\ref{lem:longpath} and~\ref{lem:cycle}).
Finally, in Step~\ref{step3}, we can finish the embedding of $F$ by using our switching technique from~\cite{BMPPMinDegree}.
The advantage is that instead of embedding into the small left over, we can embed the remaining vertices in some set $B(v)$ of linear size and then afterwards alter the embedding accordingly.

The rest of this paper is structured as follows.
In Section~\ref{sec:tools} we collect helpful tools that we will use in our proof.
Then, in Section~\ref{sec:overview}, we give a more detailed overview of the proof from Theorem~\ref{thm:girth_2-universal}, where we properly define the pseudorandom properties that we require from $G(n,p)$, the reservoir sets and also state the lemmas involving these.
In Section~\ref{sec:proof} we give the proof of Theorem~\ref{thm1}, which, together with Proposition~\ref{prop:expander} from Section~\ref{sec:overview}, implies Theorem~\ref{thm:girth_2-universal}.
Finally, in Section~\ref{sec:auxiliary}, we prove the remaining statements, Proposition~\ref{prop:expander} and Lemma~\ref{lem:2-partition} and~\ref{lem:embeddingF_1}.

\subsection{Notation and tools}
\label{sec:tools}

Throughout the paper we will use standard graph theoretic notation following~\cites{janson2011random,KF_Random}.
We collect the most relevant definitions here and give some more later.
Let $G$ and $H$ be graphs.
We denote the set of \emph{vertices} by $V(G)$ and the set of \emph{edges} by $E(G)$.
For a set $V' \subseteq V$ we denote by $G[V']$ the \emph{subgraph of $G$ induced on $V'$} and by $H \setminus G$ the subgraph of $H$ induced on $V(H) \setminus V(G)$.
Also, for $v \in V(G)$ the set of \emph{neighbours} of $v$ in $G$ is $N_{G}(v)$.
A sequence of distinct vertices $v_0,\dots,v_k$ with $v_iv_{i+1} \in E(G)$ for $i=0,\dots,k-1$ is called a \emph{path of length $k$} in $G$.
We write $d_G(u,v)$ for the \emph{distance} between two vertices $v,u \in V(G)$, which is the length of a shortest path in $G$ between them and $\infty$ if there is no path.
Furthermore, we slightly abuse notation by writing $a = b \pm c$ for $a \in [b-c,b+c]$ and $b \pm c$ for some number in the interval $[b-c,b+c]$.

We will use the following result by Ben-Eliezer, Krivelevich, and Sudakov~\cite{BKS_SizeRamseyPath} which enables us to find a long path in any graphs that has an edge between any two sets of linear size.

\begin{lem}
	\label{lem:longpath}
	Let $\varepsilon>0$ and assume that $G$ is a graph on $n$ vertices, containing an edge between any two disjoint subsets $V_1,V_2 \subseteq V$ of size $|V_1|,|V_2| \ge \varepsilon n$.
	Then $G$ contains a path of length at least $n-2 \varepsilon n$.
\end{lem}

Finding almost spanning embeddings is much easier, in particular, for factors.
The following lemma helps to add further cycles if there is still a linear number of vertices left.
\begin{lem}
	\label{lem:cycle}
	Let $\ell \ge 3$ be an integer, $\varepsilon > 0$, and $p = \omega(n^{-(\ell-1)/\ell})$.
	Then a.a.s.~in the random graph $G(n,p)$ for any choice of disjoint sets of vertices $V_1,\dots,V_\ell$ of size at least $\varepsilon n/\ell$ the number of cycles $v_1,\dots,v_\ell$ with $v_i \in V_i$ for $1\le i \le \ell$ is at least $\tfrac{1}{2} \prod_{i=1}^k p \cdot |V_i|$.
\end{lem}
This is a special case of~\cite[Theorem~4.9]{janson2011random} and the proof is a standard application of Janson's inequality (c.f.~the version in~\cite[Theorem 21.12]{KF_Random}), which proved to be a very useful tool for embedding small graphs.

\section{Overview of the proof from Theorem~\ref{thm:girth_2-universal} and~\ref{thm:large_2-universal}}
\label{sec:overview}

The general approach is similar to~\cite{BHKMPPtrees}, where we are proving universality for bounded degree spanning trees.
Here we want to embed every $F \in \cF^\ell(n,2)$ into $H=G_\alpha \cup G(n,p)$.
For proving this universality statement it is helpful to extract a deterministic graph $G$ from $G(n,p)$, which satisfies the crucial conditions for our embedding.

Let $G$ be a graph on $n$ vertices and $T$ a graph on vertices $v_1,\dots,v_t$.
We denote by $t_T(G)$ the number of copies of $T$ in $G$.
Furthermore, for vertex sets $V_1, \dots, V_k \subseteq V(G)$ we denote by $t_T(V_1,\dots,V_k)$ the number of copies of $T$ in $G$ with $v_i \in V_i$.
If there is no labelling of the vertices of $T$ specified, then we fix an arbitrary one.

\begin{defn}
\label{def:graphs}
For $\ell_0>0$, $\ell \ge 3$ integers and $p \in (0,1)$ we say that an $n$-vertex graph $G$ is an \emph{$(n, p, \ell_0, \ell)$-graph} if the following holds:
	\begin{enumerate}[label={(\bfseries A\arabic*)}]
	\item \label{def:edges} For any disjoint $V_1, V_2 \subset V(G)$ such that $|V_1|, |V_2|\ge n / \ell_0$, we have $t_{K_2}(V_1, V_2)\ge (p/2) |V_1| |V_2|$.
	\item \label{def:cherries} For any pairwise disjoint $V_1, V_2, V_3 \subset V(G)$ such that $|V_1|, |V_2|, |V_3| \ge n/\ell_0$, we have $t_{K_{1,2}}(V_1, V_2,V_3) \ge (p^2/4) |V_1| |V_2| |V_3|$.
	Furthermore, $t_{K_{1,2}}(G) \le p^2 n^3$.
	\item \label{def:cycles} For all $\ell \le k < \ell_0$ and for any pairwise disjoint sets $V_1,\dots,V_k \subseteq V(G)$ such that $|V_i| \ge n /\ell_0^2$ for $1 \le i \le k$, we have $t_{C_k}(V_1,\dots,V_k) \ge \tfrac{1}{2} p^k \prod_{i=1}^k |V_i|$.
	Furthermore, for $k=3,4$ we also have $t_{C_k}(G) \le p^k n^k$.
\end{enumerate}
\end{defn}

We denote the family of $(n, p, \ell_0, \ell)$-graphs by $\cG(n, p, \ell_0, \ell)$.
Basically, the elements from $\cG(n, p, \ell_0, \ell)$ are graphs where we can control the number of edges, cherries, and cycles (of size $\ell \le k <\ell_0$) between sets of linear size.
To cover Theorem~\ref{thm:girth_2-universal} and~\ref{thm:large_2-universal} simultaneously we define
\begin{align*}
p_{\ell_0,\ell}(n):=
\begin{cases}
	n^{-(\ell-1)/\ell} \text{ for }\ell < \ell_0,\\
	n^{-1} \text{ for } \ell \ge \ell_0.
\end{cases}
\end{align*}
The following two statements will readily imply both theorems.

\begin{prop}\label{prop:expander}
	Let $\ell_0, \ell \ge 3$ be integers and $p = \omega(p_{\ell_0,\ell}(n))$.
	Then the random graph $G(n,p)$ a.a.s.~is in $\cG(n,p,\ell_0, \ell)$.
\end{prop}

Note that $p = \omega(p_{\ell_0,\ell}) = \omega(n^{-1})$ for all $\ell,\ell_0$, which is sufficient for Condition~\ref{def:edges} and~\ref{def:cherries}, and that for $\ell \ge \ell_0$ Condition~\ref{def:cycles} is obsolete for the graphs in $\cG(n,p,\ell_0,\ell)$ and .

\begin{thm}
	\label{thm1}
	For any $\alpha>0$ there exist an $\ell_0>0$ and $n_0$ such that the following holds for any $\ell \ge 3$ and $n\ge n_0$.
	Suppose $p = \omega(p_{\ell_0,\ell}(n))$, $G \in \cG(n, p, \ell_0, \ell)$, and $G_\alpha$ an $n$-vertex on the same vertex set as $G$ with $\delta(G_\alpha) \ge \alpha n$.
	Then $H:=G_\alpha \cup G$ is $\cF^\ell(n, \Delta)$-universal.
\end{thm}

\begin{proof}[Proof of Theorem~\ref{thm:girth_2-universal} and~\ref{thm:large_2-universal}]
	Given $\alpha$, let $\ell_0$ and $n_0$ be given by Theorem~\ref{thm1}.
	Let $\ell \ge 3$ for Theorem~\ref{thm:girth_2-universal} and $\ell \ge \ell_0$ for Theorem~\ref{thm:large_2-universal}.
	We apply Proposition~\ref{prop:expander} with this $\ell$, $\ell_0$, and $p$.
	Since a.a.s.~the random graph $G(n,p)$ is in $\cG(n,p,\ell_0,\ell)$ we have, by Theorem~\ref{thm1}, that $G_\alpha\cup G(n,p)$ is  a.a.s.~$\cF^\ell(n,2)$-universal as desired.
\end{proof}

Since in Theorem~\ref{thm1} $H$ is a deterministic graph we can fix some $F \in \cF^\ell(n,2)$ and it remains to show that there is an embedding of $F$ into $H$.
Given $\alpha>0$ we will work with constants $\beta, \varepsilon >0$ and an integer $\ell_0$ such that
	\begin{align}
		\label{eq:parameters}
		20 \beta \le \alpha, \qquad \varepsilon \le 10^{-4} \alpha^3 \beta/2, \quad\text{and}\quad \ell_0 \ge 10/\varepsilon,
	\end{align}
	
For Theorem~\ref{thm1} the size of cycles, which we can embed directly into $G \in \cG(n,p,\ell_0,\ell)$, is bounded by $\ell_0$.
We decompose the vertex set of $F$ into three parts.
It is convenient to work with a collection of cycles.
For this let $\cF^\ell_*(n,2)$ be the subset of the edgewise maximal graphs from $\cF^\ell(n,2)$.
Then $F \in \cF^\ell_*(n,2)$ consists of the disjoint union of cycles and possibly one path of length $k$ with $0 \le k \le \ell-2$.

\begin{lem}
	\label{lem:2-partition}
	Let $\alpha>0$, $\ell \ge 3$ be an integer, $F \in \cF^\ell_*(n,2)$, and $\beta$, $\varepsilon$, and $\ell_0$ be such that~\eqref{eq:parameters} holds.
	Then there exist sets of vertices $U$ and $W$ with $U \subseteq W \subseteq V:=V(F)$ such that $ |U| = 10 \beta n \pm 2$, $ |V \setminus W| = \varepsilon n \pm 2$ and the following holds for the induced subgraphs:
	\begin{enumerate}[label={(\bfseries P\arabic*)}]
		\item \label{prop:pathscycles} $F[W \setminus U]$ only contains cycles of length less than $\ell_0$ and arbitrarily long paths .
		\item \label{prop:K3K2} $F[V\setminus W]$ only contains isolated edges and $C_3$'s.
		\item \label{prop:neighbours} In $F$ there are no edges between $U$ and $W \setminus U$ and at most two edges between $U$ and $V \setminus W$.
		
	\end{enumerate}
\end{lem}

While Property~\ref{prop:pathscycles} is essential for the embedding, Property~\ref{prop:K3K2} and~\ref{prop:neighbours} mostly are there to simplify the proof.
To prove this lemma we greedily partition cycles (of length at least $\ell_0$) into paths and put aside some edges.

For the proof of Theorem~\ref{thm1} we first embed $F[U]$ into $G$ with an additional reservoir property.
We now define these reservoir sets and show that we can force them to be suitably large.  
These sets will be helpful when finishing the embedding of~$F$, since they will allow us to locally alter our partial embeddings.

Let $V$ be a set of $n$ vertices.
Let $H$ be a graph on $V$ and let $F$ be a graph with $V(F)\subset V$.
For $v\in V$, let 
\begin{align*}
	B_{F,H}(v):=\big\{w\in V(F): N_F(w)\subseteq N_{H}(v)\big\}.
\end{align*}
For distinct vertices~$u$ and~$v\in V$, we define their \emph{reservoir set}~$B_{F,H}(u, v)$ as follows:
\begin{align*}
	B_{F,H}(u, v):= B_{F,H}(v)\cap N_{H}(u).
\end{align*}
Recall that the idea is that we can free up any $w\in B_{F,H}(u,v)$ used already in the embedding, by moving the vertex embedded to~$w$ to~$v$.
This then allows us to use~$w$ for embedding any remaining neighbours of the vertex embedded to~$u$.
The following lemma gives us the desired embedding of $F[U]$, which will be Step~\ref{step1} below.

\begin{lem}
	\label{lem:embeddingF_1}
	For $\alpha>0$, $\ell \ge 3$ an integer, and $\varepsilon$, $\beta$, and $\ell_0$ such that~\eqref{eq:parameters} holds, there exist $n_0$ such that the following is true for $n \ge n_0$.
	Suppose $p = \omega(p_{\ell_0,\ell}(n))$, $G \in \cG(n,p,\ell_0, \ell)$, and $G_\alpha$ an $n$-vertex graph on the same vertex as $G$ with $\delta(G) \ge \alpha n$.
	Then for any $F \in \cF^\ell_*(10 \beta n \pm 2,2)$ there exists an embedding $f$ of $F$ into $H=G \cup G_\alpha$ such that $|B_{\tilde{F},H}(u,v)| \ge 10 \varepsilon n$ for any $u,v \in V(G)$, where $\tilde{F}=f(F)$.
\end{lem}

For the proof we first find a family of roughly $\beta n$ many cherries ($K_{1,2}$'s) in $F$.
Next we embed them uniformly at random onto cherries in $G$, which ensures the reservoir property in $H$.
Then we finish the embedding by connecting the cherries to cycles in $H$.

For Step~\ref{step2} we then find an embedding of $F[W \setminus U]$, so that together we have an embedding containing all but roughly $\varepsilon n$ vertices of $F$.
This is easy, because by~Property~\ref{prop:pathscycles} we only have to embed small cycles using Property~\ref{def:cycles} of $G$, and some longer paths using Property \ref{def:edges} of $G$ and Lemma~\ref{lem:longpath}.

Finally, we can finish the embedding of $F$ in Step~\ref{step3} by using our switching technique from~\cite{BMPPMinDegree}.
For example, suppose that $u \in V(F)$ is already embedded onto $\tilde u$ in the current embedding $\tilde F$ and we want to embed a neighbour $w$ of $u$ in $F$.
We choose any uncovered vertex $\tilde v$ in $H$ and any vertex $\tilde w$ in $B_{\tilde F,H}(\tilde u, \tilde v)$.
Let $v$ be the preimage of $\tilde w$.
In the current embedding $\tilde F$ we replace $\tilde w$ by $\tilde v$ and embed $w$ onto $\tilde w$.
This gives us again a valid embedding, because by definition of $B_{\tilde F,H}(\tilde u, \tilde v)$ the pair $\tilde w \tilde u$ is an edge in $H$ and the neighbours of $\tilde w$ in $\tilde F$ are also neighbours of $\tilde v$ in $H$.
In view of Property~\ref{prop:K3K2} we will use this for the connection of two vertices with a path of length $3$ and for embedding isolated triangles.
We now give the details of the proof.

\section{Proof of Theorem~\ref{thm1}}
\label{sec:proof}
%\begin{proof}[Proof of Theorem~\ref{thm1}]
	
	Given $\alpha>0$ we choose constants $\beta, \varepsilon>0$ and an integer $\ell_0$ such that~\eqref{eq:parameters} holds.
	Let $n_0$ be large enough, suppose that $n\ge n_0$,  $G\in \cG(n, p, \ell_0 ,\ell)$, and that $G_\alpha$ is an $n$-vertex graph on $V(G)$ with $\delta(G_\alpha) \ge \alpha n$, and let $F\in \cF^\ell(n,\Delta)$.
	We want to find an embedding of $F$ into $H:=G \cup G_\alpha$.
	We add edges to $F$ while not creating a cycle of length less than $\ell$.
	Then $F \in \cF^\ell_*(n,2)$ and there is at most one path of length $k$ in $F$ with $0 \le k \le \ell-2$.
	
	\newcounter{steps}
	\medskip
	\noindent
	\textbf{ Step \refstepcounter{steps}\thesteps\label{step1}.}
	We apply Lemma~\ref{lem:2-partition} to obtain vertex sets $U \subseteq W \subseteq V=V(F)$ of size $|U| = 20 \beta n \pm 2$ and $|V \setminus W| = \varepsilon n \pm 2$ such that~\ref{prop:pathscycles},~\ref{prop:K3K2}, and~\ref{prop:neighbours} hold.
	By~\ref{prop:neighbours} there are at most two vertices in $F':=F[U]$ which have degree one.
	To get $F' \in \cF^\ell_*(20 \beta n \pm 2,2)$ we add one edge connecting these two vertices if they are at distance at least $\ell-1$.
	From Lemma~\ref{lem:embeddingF_1} we then obtain an embedding $f'$ of $F'$ into $H$ with the described property.
	After removing the edge from $F'$ that we added previously and defining $\tilde{F'} := f'(F')$ we still have $|B_{\tilde F',H}(u,v)| \ge 10 \varepsilon n$ for all $u,v \in V(H)$.
	
	\medskip
	\noindent
	\textbf{ Step \refstepcounter{steps}\thesteps\label{step2}.}
	Let $F'' := F[W \setminus U]$ and $G':= G \setminus V(\tilde F')$.
	Note that by~\ref{prop:neighbours} there are no edges between $F'$ and $F''$.
	We want to find an embedding $f''$ of $F''$ into $G'$ and start with the empty map $f''$.
	By Property~\ref{prop:pathscycles} the graph $F''$ only contains cycles of length less than $\ell_0$ and paths of arbitrary length.
	Using~\ref{def:edges} and Lemma~\ref{lem:longpath} we can find a path of length $(1-\varepsilon) n$ in $G'$.
	We extend $f''$ by embedding all paths from $F_2$ to consecutive segments from this path.

	Furthermore, for any cycle $C_k$ in $F''$ we use $|V \setminus W| \ge \varepsilon n -2 \ge n/\ell_0$ to find pairwise disjoint $V_1,\dots,V_k \subset V \setminus W$ with $|V_i| \ge n/\l_0^2$ for $1\le i \le k$. Then using~\ref{def:cycles}, we embed $C_k$ into the uncovered vertices of $G'$.
	Let $f''$ be the resulting embedding of $F''$ into $G'$.
	We combine $f'$ and $f''$ to obtain an embedding $f_0$ of $F' \cup F'' = F[W]$ into $H$.
	Let $\tilde F_0 := f_0 (F_1 \cup F_2)$ and observe that $|B_{\tilde{F}_0,H}(u,v)| \ge 10 \varepsilon n$ for all $u,v \in V(H)$, because $N_{\tilde F_0}(x)=N_{\tilde F'}(x)$ for all $x \in V(\tilde F')$.

	\medskip
	\noindent
	\textbf{ Step \refstepcounter{steps}\thesteps\label{step3}.}
	It remains to embed the $\varepsilon n \pm 2$ vertices of $V \setminus W$.
	We achieve this by using the sets $B_{\tilde F_0,H}(u,v)$ and the switching technique explained earlier.
	By~\ref{prop:K3K2} we only have to embed edges, possibly incident to vertices from $\tilde F_0$, and if $\ell=3$ isolated $C_3$'s.
	We can obtain $F$ from $F' \cup F''$ by iteratively adding a $C_3$ or a connection via two vertices.
	Let $F_0 \subseteq F_1 \subseteq \dots \subseteq F_t$ be a sequence like this with $F_0 = F' \cup F''$, $F_t = F$, $\varepsilon n/3 -1 \le t \le \varepsilon n/2 +1$, and for some $t' \in [t+1]$ the graph $F_i \setminus F_{i-1}$  is a triangle if $t' \le i  \le t$ and a single edge otherwise.
	Note that for $\ell \ge 4$ we always have $t'=t+1$.
	We claim that we can extend the embedding inductively while keeping $|B_{\tilde F_{i}',G}(u,v)|\ge |B_{\tilde F_{i-1}',G}(u,v)| - 10$ for every $i\in [t]$, where each $\tilde F_i$ is the image of $F_i$ in $G$.

	For $1 \le i \le t'-1$ let $w_1,w_2 \in V(F_i) \setminus V(F_{i-1})$ be the two new vertices added in this step.
	We assume there are vertices $u_1,u_2$ in $F_{i-1}$, which have been embedded to $\tilde{u}_j := f_{i-1}(u_j)$ for $j=1,2$, such that $u_1,w_1,w_2,u_2$ is a path in $F_i$.
	Further let $\tilde{v}_1$ and $\tilde{v}_2$ be two vertices in $V(G) \setminus V(\tilde F_{i-1})$.
	Then for $j=1,2$ we have
	\begin{align*}
		|B_{\tilde F_{i-1},G}(\tilde{u}_j, \tilde{v}_j)|\ge |B_{\tilde F_0,G}(\tilde u_j, \tilde{v}_j)| - (i-1)10 \ge 10 \varepsilon n - 10 t' \ge  2 \varepsilon n ,
	\end{align*}
	and, therefore, there are disjoint sets $V_1, V_2 \subseteq B_{\tilde F_{i-1},G}(\tilde{u}_j, \tilde v_j)$ of size at least $\varepsilon n$.
	Then, by~\ref{def:edges}, there is an edge $\tilde{w}_1 \tilde{w}_2$ in $E(G)$ with $\tilde{w}_j \in V_j$ for $j=1,2$.
	Let $v_1,v_2$ be those vertices with $f_{i-1}(v_j) = \tilde{w}_j$ for $j=1,2$.
	From the embedding $f_{i-1}$ of $F_{i-1}$ we construct the embedding $f_i$ of $F_i$ by defining $f_i(w_j) := \tilde{w}_j$, $f_i(v_j) := \tilde{v}_j$ for $j=1,2$, and $f_i(x) := f_{i-1}(x)$ for all $x \in V(F_{i-1}) \setminus \{ v_1,v_2  \}$.
	
	Basically, for $j=1,2$ we swap $v_j$ out of the current embedding and use its previous image $\tilde{w}_j$ to embed $w_j$, and embed $v_j$ to $\tilde{v}_j$ instead.
	Observe, that $f_i$ is an embedding of $F_i$ because $f_{i-1}$ was an embedding of $F_{i-1}$, $\tilde{w}_j\tilde{u}_j$ is an edge of $G_\alpha$ for $j=1,2$, $w_1w_2$ is an edge of $G$, and the neighbours of $\tilde{v_j}$ in $\tilde F_{i-1}$ are also neighbours of $\tilde{v}_j$ in $G_\alpha$ by the definition of $B_{\tilde F_{i-1},G}(\tilde{u}_j, \tilde{v}_j)$ for $j=1,2$.
	Let $\tilde{F}_i := f_i(F_i)$.
	Note that $N_{\tilde F_i}(x) = N_{\tilde F_{i-1}}(x)$ for all but at most $10$ vertices $x$ in $V(\tilde F_{i-1})$, the vertices $\tilde v_1, \tilde v_2, \tilde u_1, \tilde u_2$ and the neighbours of $\tilde w_1, \tilde w_2$ in $\tilde F'_{i-1}$, because these are the vertices that are incident to the edges in $E(\tilde F_i')\setminus E(\tilde F_{i-1})$.
	Thus, we have $|B_{\tilde F_{i},H}(u,v)|\ge |B_{\tilde F_{i-1},H}(u,v)| - 10$, for any $u,v\in V(H)$.
	If for some $j \in \{1,2\}$ there is no vertex $u_j$ in $V(F_{i-1})$ with $u_jw_j \in E(F_i)$, we choose any vertex $u_j \in V(F_i)$, proceed as above, and then delete the edge $\tilde w_j \tilde u_j$ afterwards.
	
	In the case $\ell=3$ for $t' \le i \le t$ we need to embed the triangle on vertices $w_1,w_2,w_3 \in V(F_i) \setminus V(F_{i-1})$.
	Let $\tilde v_1,\tilde v_2, \tilde v_3$ be three vertices in $V(H) \setminus V(\tilde F_{i-1})$.
	Then for $j=1,2,3$ we have for any $u \in V(H)$
	\begin{align*}
		|B_{\tilde F_{i-1},H}(\tilde{v}_j)| \ge |B_{\tilde F_{0},H}(u, \tilde{v}_j)| - (i-1)10 \ge 10 \varepsilon n - 10 t \ge 3 \varepsilon n,
	\end{align*}
	and, therefore, there are disjoint sets $V_1, V_2, V_3 \subseteq B_{\tilde F_{i-1},H}(\tilde v_j)$ of size at least $\varepsilon n$.
	Then, by~\ref{def:cycles}, there is a triangle  $\tilde{w}_1 \tilde{w}_2 \tilde{w}_3$ in $G$ with $\tilde{w}_j \in V_j$ for $j=1,2,3$.
	Let $v_1,v_2,v_3$ be those vertices with $f_{i-1}(v_j) = \tilde{w}_j$ for $j=1,2,3$.
	From the embedding $f_{i-1}$ of $F_{i-1}$ we construct the embedding $f_i$ of $F_i$ by defining $f_i(w_j) := \tilde{w}_j$, $f_i(v_j) := \tilde{v}_j$ for $j=1,2,3$, and $f_i(x) := f_{i-1}(x)$ for all $x \in V(F_{i-1}) \setminus \{ v_1,v_2 ,v_3 \}$.
	
	As before, we swap $v_j$ out of the current embedding for $j=1,2,3$ and use its previous image $\tilde{w}_j$ to embed $w_j$, and embed $v_j$ to $\tilde{v}_j$ instead.
	Observe, that $f_i$ is an embedding of $F_i$ because $f_{i-1}$ was an embedding of $F_{i-1}$, $\tilde{w}_1 \tilde{w}_2 \tilde{w}_3$ is a triangle in $G$ and the neighbours of $\tilde{v_j}$ in $F_{i-1}$ are also neighbours of $\tilde{v}_j$ in $G_\alpha$ for $j=1,2,3$.
	Let $\tilde{F}_i := f_i(F_i)$.
	Note that $N_{\tilde F_i}(x) = N_{\tilde F_{i-1}}(x)$ for all but at most $9$ vertices $x$ in $V(\tilde F_{i-1})$, the vertices $\tilde v_1, \tilde v_2, \tilde v_3$ and the neighbours of $\tilde w_1, \tilde w_2, \tilde w_3$ in $\tilde F_{i-1}$, because they are the vertices that are incident to the edges in $E(\tilde F_i)\setminus E(\tilde F_{i-1})$.
	Thus, we have $|B_{\tilde F_{i},H}(u,v)|\ge |B_{\tilde F_{i-1},H}(u,v)| - 9$, for any $u,v\in V(H)$.
	
	Finally, $f:=f_t$ is a spanning embedding of $F$ into $H$ and we finished the proof.
	\qed
%\end{proof}

\section{Proof of auxiliary statements}
\label{sec:auxiliary}

It remains to prove Proposition~\ref{prop:expander} and Lemma~\ref{lem:2-partition} and~\ref{lem:embeddingF_1}.
We first show that $G(n,p)$ a.a.s.~satisifes Properties~\ref{def:edges}-\ref{def:cycles} required by $\cG(n,p,\ell_0,\ell)$.

\begin{proof}[Proof of Proposition~\ref{prop:expander}]
	Let $G$ be a graph drawn from the distribution $G(n,p)$.
	By a simple Chernoff bound (c.f.~\cite[Theorem~2.8]{janson2011random}), the probability that, for all $V_1,V_2\subset V(G)$, with $|V_1|,|V_2|\ge n/\ell_0$, we have
	\begin{align}\label{edgecount}
		(p/2) |V_1||V_2| \le t_{K_2}(V_1,V_2)\le 2 p |V_1||V_2|.
	\end{align}
	is at least $1-2^{2n}e^{-\omega(n)} = 1 - o(1)$, because $p_{\ell_0,\ell} = \omega(n^{-1})$.
	So we can assume that~\ref{def:edges} holds in $G$.
	
	Next, let $V_1,V_2,V_3 \subseteq V(G)$ be disjoint sets with $|V_1|,|V_2|,|V_3| \ge n/\ell_0$.
	Then by~\eqref{edgecount} we have $t_{K_2}(V_i, V_j)\ge (p/2)|V_i||V_j|$ for all $i \not= j$.
	Then, by convexity, the number of cherries with centre in $V_1$ and neighbours in $V_2$ and $V_3$ respectively is at least
	\begin{align*}
	\sum_{v\in V_1}\deg_{V_2}(v) \cdot \deg_{V_3}(v) \ge |V_1| \left( \sum_{v \in V_1} \deg_{V_2}(v) / |V_1| \right) \cdot \left( \sum_{v \in V_1} \deg_{V_2}(v) / |V_1| \right) \ge (p^2/4) |V_1| |V_2|V_3|.
	\end{align*}
	A simple second moment calculation implies that a.a.s.~the number of cherries in $G$ is at most $p^2 n^3$ and so~\ref{def:cherries} holds.
	Note that alternatively this also follows from Janson's inequality.
	
	Finally, if $\ell<\ell_0$ we have $p_{\ell_0,\ell} = \omega(n^{-(\ell-1)/\ell})$.
	Then for any $\ell \le k < \ell_0$, let $V_1,\dots,V_k \subseteq V(G)$ be disjoint sets such that $|V_i| \ge n/\ell_0^2$ for $1 \le i \le k$.
	By an application of Lemma~\ref{lem:cycle} a.a.s.~the number of cycles $C_k$ in $G$ with one vertex in each $V_i$  is at least $(p^k/2) \prod_{i=1}^k |V_i|$.
	And again by simple second moment calculations we get that a.a.s.~the number of $C_3$ and $C_4$ in $G$ is at most $p^3n^3$ and $p^4n^4$ respectively.
	As this implies that also~\ref{def:cycles} holds a.a.s., the lemma is proved.
\end{proof}

Next, for an almost $2$-regular graph we obtain a decomposition such that~\ref{prop:pathscycles}-\ref{prop:neighbours} hold.

\begin{proof}[Proof of Lemma~\ref{lem:2-partition}]
	Let $F \in \cF^\ell_*(n,2)$. 
	We start with $U = \emptyset$ and greedily add the vertex sets of cycles from $F$ to $U$.
	We stop once $|U|$ passes $10 \beta n$ and then remove exactly $2$ ($v_1$ and $v_2$) or at least $5$ ($v_1,\dots,v_k$ with $v_iv_{i+1} \in E(F)$) vertices from the last cycle sucht that $|U|=10 \beta n \pm 2$.
	
	Now let $W = V$.
	First we want to ensure Property~\ref{prop:neighbours}.
	By the previous step there are at most two vertices in $U$ with neighbours outside of $F[U]$.
	If we removed only $v_1,v_2$ from $U$ than we also remove these two from $W$, which gives us one $K_2$ in $F[V \setminus W]$.
	Otherwise we remove the pairs $v_1,v_2$ and $v_{k-1},v_k$ from $W$, which gives us two $K_2$'s in $F[V \setminus W]$, because $v_2$ and $v_{k-1}$ are not connected by an edge.
	
	To ensure Property~\ref{prop:pathscycles} for any cycle in $F[W \setminus U]$ of length larger than $C$ we remove two neighbouring vertices from $W$.
	There are at most $n/C$ cycles of length $C$ and as $10/C \le \varepsilon$ we have $|V \setminus W| \le 2 n/C + 4 \le \varepsilon n$.
	We continue by removing the vertex sets of $K_3$'s from $F[W \setminus U]$, keeping $|V \setminus W| \le \varepsilon n + 2$.
	If we still have $|V \setminus W| < \varepsilon n - 2$ then we remove additional $K_2$'s from $F[W \setminus U]$ until $|V \setminus W| = \varepsilon n \pm 2$.
	To ensure Property~\ref{prop:K3K2} we only choose those $K_2$, which are not connected to anything else from $V \setminus W$.
	This is possible because $\varepsilon \le 1/20$ and finishes our proof.	
\end{proof}

Finally, we arrive at the key ingredient of the proof.
We embed a small fraction of the graph such that the sets $B_{\tilde{F},H}(u,v)$ are large for all $u,v \in V(H)$.
We stress that it is crucial for universality that we can do this for all choices of $F$.

\begin{proof}[Proof of Lemma~\ref{lem:embeddingF_1}]
	
	For $\alpha>0$ and $\ell \ge 3$ let $\beta$, $\varepsilon$, and $\ell_0$ such that \eqref{eq:parameters} holds.
	Further let $n_0$ be large enough, $n \ge n_0$, $p=\omega(p_{\ell_0,\ell}(n))$, $G \in \cG(n,p,\ell_0,\ell)$, $G_\alpha$ an $n$-vertex graph on the same vertex set as $G$ with $\delta(G) \ge \alpha n$, $H:= G_\alpha \cup G(n,p)$, and $F \in \cF^\ell_*(10 \beta n \pm 2,2)$.
	
	We greedily choose vertices $x_1,x_2,\dots,x_t$ in $F$ of degree $2$ with distance $d_F(x_i,x_j) \ge 5$ for all $1 \le i<j \le t$ where $t = \beta n \pm 1$ is an integer.
	For $\ell=3$ we additionally require that either all of them are contained in a $C_3$ or none.
	We first want to embed these vertices together with their neighbours.
	If $\ell \ge 5$ we find a set of disjoint cherries in $H$ where we can embed them to.
	But if $\ell<5$ some of the vertices might be contained in a $C_3$ or $C_4$ and then we have to embed the whole cycle at once as we can not hope to close it later.
	If $\ell=4$ there are no $C_3$'s and, therefore, we can embed $x_1,\dots,x_t$ with their neighbours onto a set of disjoint $C_4$'s and close the cycle if necessary.
	Finally, if $\ell=3$ than either all $x_i$ are contained in a $C_3$ or none.
	In the former case we embed these onto a set of disjoint $C_3$'s and in the latter case we proceed as for $\ell=4$.
	We call a graph \emph{centred} if it has one vertex indicated as its centre and refer to the neighbours of the centre as neighbours of the graph.
	
	\begin{claim*}
		Let $T$ be a cherry, if $\ell \le4$ a $C_4$ or if $\ell=3$ a $C_3$.
		Then there is a choice of $t$ disjoint centred copies $T_1,\dots,T_t$ in $H$ such that the following holds.
		For any $u, v\in V$ there are at least $2 \eps n$ copies of $T$ in $T_1,\dots,T_t$ with their centres in $N_{G_\alpha}(u)$ and the neighbours in $N_{G_{\alpha}}(v)$.
	\end{claim*}

	\begin{claimproof}[Proof of the Claim]
		We randomly and sequentially want to pick $t$ centred copies $T_1,\dots,T_t$ of $T$ from $G$, where each $T_i$ is picked uniformly at random from the copies of $T$ which are disjoint from $T_1,\ldots,T_{i-1}$.
		
		For $u,v\in V(H)$, $i\in [t]$, let $Y_{i}^{u,v}$ be the Bernoulli random variable for the event that $\tilde{x}_i\in N_{G_\alpha}(u)$ and $R_i \subseteq N_{G_{\alpha}}(v)$, where $\tilde{x}_i$ is the centre of $T_i$ and $R_i$ is the set containing the two neighbours of $\tilde{x}_i$ in $T_i$.
		Since $\delta(G_\alpha) \ge \alpha n$, $|F| = 20 \beta n \pm 2$ and the existing copies of $T$ cover at most $4 t \le 4 \beta n + 4 \le \alpha n/8$ vertices, there are at least $7\alpha n/8$ vertices available in both $N_{G_\alpha}(u)\setminus \bigcup_{j\in [i-1]} V(T_j)$ and $N_{G_\alpha}(v)\setminus \bigcup_{j\in [i-1]} V(T_j)$.
		Therefore, there are sets $V_1 \subseteq N_{G_\alpha}(u)\setminus \bigcup_{j\in [i-1]} V(T_j)$, $V_2,V_3 \subseteq N_{G_\alpha}(v)\setminus \bigcup_{j\in [i-1]} V(T_j)$, and $V_4 \subseteq V(H) \setminus \bigcup_{j\in [i-1]} V(T_j)$ of size $|V_i| = \alpha n/4$ for $i=1,2,3,4$.
		We now consider the three cases, where $T$ is a cherry, $C_3$, or $C_4$.
		
		\begin{itemize}[leftmargin=17pt]
			\item
				In the case $T = K_{1,2}$ we will find cherries with the centre in $V_1$ and the neighbours in $V_2$ and $V_3$.
				Since $G \in \cG(n, p, \ell_0, \ell)$ and $\alpha /4 \ge 1/\ell_0$ by~\ref{def:cherries} the number of cherries that we are interested in is at least $\alpha^3 p^2 n^3/256$.
				The total number of cherries in $G$ is at most $p^2 n^3$.
				This allows us to obtain
			\begin{align*}
				\mathbb{E} (Y_i^{u,v}\mid Y_1^{u,v},\dots, Y_{i-1}^{u,v}) \ge \frac{\alpha^3 p^2 n^3/256}{p^2 n^3}\ge \alpha^3 /256.
			\end{align*}
			\item
				When $T=C_4$ we embed the centre into $V_1$, both neighbours into $V_2$, $V_3$, and the last vertex into $V_4$.
				As $\ell \le4$, $G \in \cG(n, p, \ell_0, \ell)$, and $\alpha /4 \ge 1/\ell_0$ by~\ref{def:cycles} there are at least $\alpha^4 p^4 n^4/128$ suitable copies of $C_4$.
				On the other hand the total number of $C_4$ is at most $p^4 n^4$, which gives us
			\begin{align*}
				\mathbb{E} (Y_i^{u,v}\mid Y_1^{u,v},\dots, Y_{i-1}^{u,v}) \ge \frac{\alpha^4 p^4 n^4/128}{p^4 n^4}\ge \alpha^4 /512.
			\end{align*}
			\item 
				Finally, for $T=C_3$ we embed the centre into $V_1$ and the other two vertices into $V_2$ and $V_3$.
				As $\ell = 3$, we obtain from~\ref{def:cycles} that there are at least $\alpha^3 p^3 n^3/32$ suitable copies of $C_3$ and at most $p^3 n^3$ copies of $C_3$ in $G$ in total.
				This implies
			\begin{align*}
				\mathbb{E} (Y_i^{u,v}\mid Y_1^{u,v},\dots, Y_{i-1}^{u,v}) \ge \frac{\alpha^3 p^3 n^3/32}{p^3 n^3}\ge \alpha^3 /128.
			\end{align*}
		\end{itemize}

	Let $x:= t \alpha^4 /512 \ge \beta \alpha^4 n /1000 \ge 20 \varepsilon n$ by the choice of $\eps$ in \eqref{eq:parameters}.
	Thus, by~\cite[Lemma~2.2]{allen2016blow} (the sequential dependence lemma) with $\delta=1/2$, or a simple coupling argument, we get
	\begin{equation*}
		\mathbb{P}\big(Y_1^{u,v}+ \cdots + Y_{t}^{u,v} < 10\eps n \big) \le \mathbb{P}\big(Y_1^{u,v}+ \cdots + Y_{t}^{u,v} < x/2 \big) < e^{- x/12} \le e^{-\eps n}\,.
	\end{equation*}
	With a union bound, we conclude that there is a choice of $T_1,\dots,T_t$ such that, for each $u, v\in V(H)$, $Y_1^{u,v}+ \cdots + Y_{s}^{u,v} \ge 10\eps n$, i.e., the claim holds for any of the $T$.
	\end{claimproof}
	
	With $T_1,\dots,T_t$ as given by the claim we define the embedding $f$ of the graphs $T$ centred at $x_1,\dots,x_t$ by mapping $x_i$ to the centre $\tilde{x}_i$ of $T_i$ and the remaining vertices of this copy of $T$ accordingly.
	Note that for $\ell \le 4$ we might embed a cherry onto a $C_4$ and leave one vertex uncovered.
	This gives us a partial embedding $f$ of $F$.

	Next we extend this embedding $f$ by embedding additional components of $F$ and extending/connecting existing parts.
	For $\ell \le 4$ we first embed all $C_3$'s and $C_4$'s from $F$ which do not contain a vertex from $x_1,\dots,x_t$.
	Let $T$ be a $C_3$ or $C_4$ as a subgraph of $F_1$ which we have not yet embedded.
	As $|F| \le 10 \beta n +2 \le \alpha n/2$ and $G \in \cG(n,p,\ell_0,\ell)$ by \ref{def:cycles} there is a copy $\tilde{T}$ of $T$ in the vertices of $G$ not covered by $f$.
	We extend $f$ by embedding $T$ onto $\tilde{T}$.
	
	If there is some component of $G$ which has not been touched by our embedding so far, we choose an arbitrary vertex $v$ and extend $f$ by embedding it to an arbitrary available vertex $\tilde v$.
	To extend the existing parts let $u$ be any vertex of $F$, which is not yet embedded and has exactly one neighbour $v$ and at most one vertex at distance two that are already embedded (the second condition ensures that we do not close the gap to much).
	Then as $|N_{G_\alpha}(f(v))| \ge \alpha n$ and $|F| \le 10 \beta n + 2 \le \alpha n/2$ there are at least $\alpha n/2$ choices in $N_{G_\alpha}(f(v))$ for the image $\tilde{u}$ of $u$.
	We choose one arbitrarily and define $f(u) = \tilde{u}$.
	
	Finally, we want to finish the embedding of $F$ by connecting all existing paths to cycles.
	Oberserve that by our choice of $x_1,\dots,x_t$ and the previous step, all vertices that are not embedded lie on a path of length $3$ connecting two vertices which are already embeded by $f$.
	So let $f(u_1)=\tilde{u}_1$ and $f(u_2)=\tilde{u}_2$ be two vertices (which have degree one in the image of $f$) such that from the path $u_1,v_1,v_2,u_2$ in $F$ both vertices $u_1,u_2$ are not yet embedded.
	Let $V_1$ be the available vertices in $N_{G_\alpha}(\tilde{v}_1)$ and $V_2$ the available vertices in $N_{G_\alpha}(\tilde{v}_2)$ after removing the image of $f$.
	As before we have $|V_1|,|V_2| \ge \alpha n/2$ and using~\ref{def:edges} we find two vertices $\tilde{u}_1$ and $\tilde{u}_2$ such that $\tilde{v}_1\tilde{u}_1 \in E(G_\alpha)$, $\tilde{u}_1\tilde{u}_2 \in E(G)$, and $\tilde{u}_2\tilde{u}_2 \in E(G_\alpha)$.
	We extend $f$ by defining $f(u_1)=\tilde{u}_1$ and $f(u_2)=\tilde{u}_2$.
	We repeat the above until all vertices are embedded.
	For the final embedding $f$ of $F$ into $H$ let $\tilde{F}=f(F)$.
	
	By the claim for any $u, v\in V$, there are at least $10 \eps n$
	graphs from $T_1,\dots, T_t$ such that their centres are in
	$N_{G_\alpha}(u)$ and the neighbours are in $N_{G_{\alpha}}(v)$.
	Since these graphs are contained in $\tilde{F}$, we conclude that $|B_{\tilde{F},H}(u,v)|\ge 10\eps n$ for any $u, v\in V$, as required.
\end{proof}

%\bibliographystyle{amsplain_yk}
%\bibliography{literatur}

\begin{bibdiv}
\begin{biblist}

\bib{AB_maxdegree2}{article}{
      author={Aigner, M.},
      author={Brandt, S.},
       title={Embedding arbitrary graphs of maximum degree two},
        date={1993},
     journal={{Journal of the London Mathematical Society}},
      volume={48},
       pages={39\ndash 51},
}

\bib{allen2016blow}{article}{
      author={Allen, Peter},
      author={B{\"o}ttcher, Julia},
      author={H{\`a}n, Hiep},
      author={Kohayakawa, Yoshiharu},
      author={Person, Yury},
       title={Blow-up lemmas for sparse graphs},
        date={2016},
     journal={arXiv:1612.00622},
       pages={122 pages},
}

\bib{BTW_tilings}{article}{
      author={Balogh, J\'{o}zsef},
      author={Treglown, Andrew},
      author={Wagner, Adam~Zsolt},
       title={Tilings in randomly perturbed dense graphs},
        date={2018},
     journal={Combinatorics, Probability and Computing},
       pages={1–18},
}

\bib{BHKM_powers}{article}{
      author={Bedenknecht, Wiebke},
      author={Han, Jie},
      author={Kohayakawa, Yoshiharu},
      author={Mota, Guilherme~Oliveria},
       title={Powers of tight {H}amilton cycles in random perturbed
  hypergraphs},
        date={2018},
     journal={arXiv:1802.08900},
       pages={13 pages},
}

\bib{BKS_SizeRamseyPath}{article}{
      author={Ben-Eliezer, Ido},
      author={Krivelevich, Michael},
      author={Sudakov, Benny},
       title={The size ramsey number of a directed path},
        date={2012},
        ISSN={0095-8956},
     journal={Journal of Combinatorial Theory, Series B},
      volume={102},
      number={3},
       pages={743 \ndash  755},
  url={http://www.sciencedirect.com/science/article/pii/S0095895611001006},
}

\bib{bohman2004adding}{article}{
      author={Bohman, Tom},
      author={Frieze, Alan~M.},
      author={Krivelevich, Michael},
      author={Martin, Ryan~R.},
       title={Adding random edges to dense graphs},
        date={2004},
     journal={Random Structures {\&} Algorithms},
      volume={24},
      number={2},
       pages={105\ndash 117},
         url={https://doi.org/10.1002/rsa.10112},
}

\bib{bohman2003many}{article}{
      author={Bohman, Tom},
      author={Frieze, Alan~M.},
      author={Martin, Ryan~R.},
       title={How many random edges make a dense graph {H}amiltonian?},
        date={2003},
     journal={Random Structures {\&} Algorithms},
      volume={22},
      number={1},
       pages={33\ndash 42},
         url={https://doi.org/10.1002/rsa.10070},
}

\bib{BolTho87}{article}{
      author={Bollob{\'{a}}s, B{\'{e}}la},
      author={Thomason, Andrew},
       title={Threshold functions},
        date={1987},
     journal={Combinatorica},
      volume={7},
      number={1},
       pages={35\ndash 38},
         url={https://doi.org/10.1007/BF02579198},
}

\bib{BHKMPPtrees}{article}{
      author={B{\"{o}}ttcher, Julia},
      author={Han, Jie},
      author={Kohayakawa, Yoshiharu},
      author={Montgomery, Richard},
      author={Parczyk, Olaf},
      author={Person, Yury},
       title={Universality for bounded degree spanning trees in randomly
  perturbed graphs},
        date={2018},
     journal={arXiv:1802.04707},
       pages={12 pages},
        note={Accepted for publication in Random Structures \& Algorithms},
}

\bib{BMPPMinDegree}{article}{
      author={B{\"{o}}ttcher, Julia},
      author={Montgomery, Richard},
      author={Parczyk, Olaf},
      author={Person, Yury},
       title={Embedding spanning bounded degree subgraphs in randomly perturbed
  graphs},
        date={2018},
     journal={arXiv:1802.04603},
       pages={25 pages},
}

\bib{DT_ramsey}{article}{
      author={Das, Shagnik},
      author={Treglown, Andrew},
       title={Ramsey properties of randomly perturbed graphs: cliques and
  cycles},
        date={2019},
     journal={arXiv:1901.01684},
       pages={23 pages},
}

\bib{DKRR15}{article}{
      author={{Dellamonica Jr.}, Domingos},
      author={Kohayakawa, Yoshiharu},
      author={R{\"{o}}dl, Vojtech},
      author={Ruci{\'n}ski, Andrzej},
       title={An improved upper bound on the density of universal random
  graphs},
        date={2015},
     journal={Random Structures {\&} Algorithms},
      volume={46},
      number={2},
       pages={274\ndash 299},
         url={https://doi.org/10.1002/rsa.20545},
}

\bib{dirac1952some}{article}{
      author={Dirac, Gabriel~Andrew},
       title={Some theorems on abstract graphs},
        date={1952},
     journal={Proceedings of the London Mathematical Society},
      volume={3},
      number={1},
       pages={69\ndash 81},
}

\bib{ferber2016optimal}{article}{
      author={Ferber, Asaf},
      author={Kronenberg, Gal},
      author={Luh, Kyle},
       title={Optimal threshold for a random graph to be 2 universal},
        date={2016},
     journal={arXiv:1612.06026},
       pages={23 pages},
}

\bib{ferber2016embedding}{article}{
      author={Ferber, Asaf},
      author={Luh, Kyle},
      author={Nguyen, Oanh},
       title={Embedding large graphs into a random graph},
        date={2017},
        ISSN={1469-2120},
     journal={Bulletin of the London Mathematical Society},
      volume={49},
      number={5},
       pages={784\ndash 797},
         url={http://dx.doi.org/10.1112/blms.12066},
}

\bib{FerbeNenadovSpanning}{article}{
      author={Ferber, Asaf},
      author={Nenadov, Rajko},
       title={Spanning universality in random graphs},
     journal={Random Structures \& Algorithms},
      volume={53},
      number={4},
       pages={604\ndash 637},
}

\bib{KF_Random}{book}{
      author={Frieze, Alan},
      author={Karo{\'n}ski, Micha{\l}},
       title={Introduction to random graphs},
   publisher={Cambridge University Press},
        date={2016},
}

\bib{HZ_perturbedcycles}{article}{
      author={Han, Jie},
      author={Zhao, Yi},
       title={Hamiltonicity in randomly perturbed hypergraphs},
        date={2018},
     journal={arXiv:1802.04586},
       pages={16 pages},
}

\bib{janson2011random}{book}{
      author={Janson, Svante},
      author={Łuczak, Tomasz},
      author={Ruci{\'n}ski, Andrzej},
       title={Random graphs},
   publisher={John Wiley {\&} Sons},
        date={2000},
}

\bib{JohanssonKahnVu_FactorsInRandomGraphs}{article}{
      author={Johansson, Anders},
      author={Kahn, Jeff},
      author={Vu, Van~H.},
       title={Factors in random graphs},
        date={2008},
     journal={Random Structures {\&} Algorithms},
      volume={33},
      number={1},
       pages={1\ndash 28},
         url={https://doi.org/10.1002/rsa.20224},
}

\bib{JK_Trees}{article}{
      author={Joos, Felix},
      author={Kim, Jawhoon},
       title={Spanning trees in randomly perturbed graphs},
        date={2018},
     journal={arXiv:1803.04958},
       pages={41 pages},
}

\bib{KSS_AlonYuster}{article}{
      author={Koml{\'o}s, J{\'a}nos},
      author={S{\'a}rk{\"o}zy, G{\'a}bor~N.},
      author={Szemer{\'e}di, Endre},
       title={Proof of the {A}lon-{Y}uster conjecture},
        date={2001},
     journal={Discrete Mathematics},
      volume={235},
      number={1-3},
       pages={255\ndash 269},
         url={https://doi.org/10.1016/S0012-365X(00)00279-X},
}

\bib{Kor76}{article}{
      author={Kor{\v{s}}unov, A.~D.},
       title={Solution of a problem of {P}. {E}rd{\H o}s and {A}. {R}\'enyi on
  {H}amiltonian cycles in undirected graphs},
        date={1976},
     journal={Doklady Akademii Nauk SSSR},
      volume={228},
      number={3},
       pages={529\ndash 532},
}

\bib{krivelevich2015cycles}{article}{
      author={Krivelevich, Michael},
      author={Kwan, Matthew},
      author={Sudakov, Benny},
       title={Cycles and matchings in randomly perturbed digraphs and
  hypergraphs},
        date={2016},
     journal={Combinatorics, Probability {\&} Computing},
      volume={25},
      number={6},
       pages={909\ndash 927},
         url={https://doi.org/10.1017/S0963548316000079},
}

\bib{krivelevich2015bounded}{article}{
      author={Krivelevich, Michael},
      author={Kwan, Matthew},
      author={Sudakov, Benny},
       title={Bounded-degree spanning trees in randomly perturbed graphs},
        date={2017},
     journal={{SIAM} Journal on Discrete Mathematics},
      volume={31},
      number={1},
       pages={155\ndash 171},
         url={https://doi.org/10.1137/15M1032910},
}

\bib{MM_HyperPeturbed}{article}{
      author={McDowell, Andrew},
      author={Mycroft, Richard},
       title={{H}amilton {$\ell$}-cycles in randomly-perturbed hypergraphs},
        date={2018},
     journal={The Electronic Journal of Combinatorics},
      volume={25},
      number={4},
       pages={P4.36},
}

\bib{M14b}{article}{
      author={Montgomery, R.},
       title={Embedding bounded degree spanning trees in random graphs},
        date={2014},
     journal={arXiv:1405.6559},
       pages={14 pages},
}

\bib{M14a}{article}{
      author={Montgomery, R.},
       title={Spanning trees in random graphs},
        date={2018},
     journal={arXiv:1810.03299},
}

\bib{NT_sprinkling}{article}{
      author={Nenadov, Rajko},
      author={Trujić, Miloš},
       title={Sprinkling a few random edges doubles the power},
        date={2018},
     journal={arXiv:1811.09209},
       pages={18 pages},
}

\bib{Pos76}{article}{
      author={P{\'{o}}sa, L.},
       title={{H}amiltonian circuits in random graphs},
        date={1976},
     journal={Discrete Mathematics},
      volume={14},
      number={4},
       pages={359\ndash 364},
         url={https://doi.org/10.1016/0012-365X(76)90068-6},
}

\bib{RSR08_Dirac}{article}{
      author={R{\"{o}}dl, Vojtech},
      author={Szemer{\'{e}}di, Endre},
      author={Rucinski, Andrzej},
       title={An approximate dirac-type theorem for \emph{k} -uniform
  hypergraphs},
        date={2008},
     journal={Combinatorica},
      volume={28},
      number={2},
       pages={229\ndash 260},
         url={https://doi.org/10.1007/s00493-008-2295-z},
}

\end{biblist}
\end{bibdiv}


\end{document}